\documentclass[12pt,a4paper]{amsart}
\usepackage{amsfonts}
\usepackage{amsthm}
\usepackage{amsmath}
\usepackage{amscd}
\usepackage[latin2]{inputenc}
\usepackage{t1enc}
\usepackage[mathscr]{eucal}
\usepackage{indentfirst}
\usepackage{graphicx}
\usepackage{graphics}
\usepackage{pict2e}
\usepackage[T1]{fontenc}
\usepackage{comment}
\usepackage{currvita}
\usepackage{bbm}
\usepackage{url}
\usepackage{xcolor}
\usepackage{epic}
\usepackage{cleveref}
\numberwithin{equation}{section}
\usepackage[margin=2.9cm]{geometry}
\usepackage{epstopdf}

\theoremstyle{plain}
\newtheorem{Th}{Theorem}[section]
\newtheorem{lem}[Th]{Lemma}
\newtheorem{clm}[Th]{Claim}

\newtheorem{Prop}[Th]{Proposition}

 \theoremstyle{definition}

\newtheorem{?}[Th]{Problem}

\newtheorem{Obs}[Th]{Observation}

\newcommand{\PP}{\mathbb{P}}

\newcommand{\vd}[1]{{\color{green!50!black}{#1}}}

\begin{document}
\author{Vojt\u{e}ch Dvo\u{r}\'ak}
\address[Vojt\u{e}ch Dvo\u{r}\'ak]{Department of Pure Maths and Mathematical Statistics, University of Cambridge, UK}
\email[Vojt\u{e}ch Dvo\u{r}\'ak]{vd273@cam.ac.uk}

\author{Peter van Hintum}
\address[Peter van Hintum]{Department of Pure Maths and Mathematical Statistics, University of Cambridge, UK}
\email[Peter van Hintum]{pllv2@cam.ac.uk}

\author{Marius Tiba}
\address[Marius Tiba]{Department of Pure Maths and Mathematical Statistics, University of Cambridge, UK}
\email[Marius Tiba]{mt576@cam.ac.uk}

\title{Improved bound for Tomaszewski's problem
}


\begin{abstract} 
In 1986, Tomaszewski made the following conjecture. Given $n$ real numbers $a_{1},...,a_{n}$ with $\sum_{i=1}^{n}a_{i}^{2}=1$, then of the $2^{n}$ signed sums $\pm a_{1} \pm ... \pm a_{n}$, at least half have absolute value at most $1$. Hendriks and Van Zuijlen (2020) and Boppana (2020) independently proved that a proportion of at least $0.4276$ of these sums has absolute value at most $1$. Using different techniques, we improve this bound to $0.46$.
\end{abstract}

\maketitle

\section{Introduction} 

Take $a_{1},...,a_{n}$ real numbers such that $\sum_{i=1}^{n}a_{i}^{2}=1$ and consider the randomly signed sum $\sum_{i=1}^n \epsilon_i a_i$, where the $\epsilon_i$ are independent, identically distributed (i.i.d.) Rademacher random variables, i.e. $\PP(\epsilon_i=1)=\PP(\epsilon_i=-1)=\frac12$. In 1986, Tomaszewski (see \cite{guy}) conjectured that $\PP(|\sum_{i=1}^n \epsilon_i a_i|\leq 1)\geq \frac12$. Note that this bound is tight for $n \ge 2$ as we can take for instance $a_{1}=a_{2}=\frac{1}{\sqrt{2}}$, $a_i=0$ for $2 < i \le n$. While various partial results towards this conjecture were proven, the original problem is still open.

Several papers have focussed on showing bounds from below approaching $1/2$. Holzman and Kleitman \cite{holzman} proved that $\PP(|\sum_{i=1}^n \epsilon_i a_i|\leq 1)\geq \frac38$. In fact, they showed the stronger, tight result that $\PP(|\sum_{i=1}^n \epsilon_i a_i| < 1)\geq \frac38$ as long as there is more than one non-zero term. Later, but independently and using different techniques, Ben-Tal, Nemirovski and Roos \cite{bental} obtained the weaker bound of $\frac{1}{3}$. Their method was later refined by Shnurnikov \cite{shnurnikov} to obtain the bound of $0.36$, still weaker than the result of Holzman and Kleitman.

More recently, Boppana and Holzman \cite{boppanaold} obtained a bound of $0.406259$. Using a result of Bentkus and Dzindzalieta \cite{bentkus}, their argument can be improved to actually give a better bound of approximately $0.4276$, as was independently observed by Hendriks and Van Zuijlen \cite{zuijlen} and Boppana \cite{boppananew}. We make further progress on Tomaszewski's conjecture by using different techniques to prove our main theorem.

\begin{Th} 
Let $a_{1},...,a_{n}$ be real numbers such that $\sum_{i=1}^{n}a_{i}^{2}=1$ and let $\epsilon_{i}$ for $i=1,...,n$ be i.i.d. random variables with $\mathbb{P}(\epsilon_{i}=+1)=\mathbb{P}(\epsilon_{i}=-1)= \frac{1}{2}$. Then $\PP(|\sum_{i=1}^n \epsilon_i a_i|\leq 1)\geq 0.46.$
\end{Th}

Note that partial sums $\sum_{i=1}^k \epsilon_ia_i$ can be interpreted as a random walk with prescribed step sizes. 
This interpretation suggests common techniques like mirroring, symmetry and second moment arguments, as have been used in previous papers on this problem \cite{bental,boppananew,boppanaold,zuijlen,holzman,shnurnikov}. We manage to set up a framework which allows for a tight interplay between all these techniques, by combining them with ideas from linear programming.

Depending on the size of $\max\{|a_i|\}$, we consider four cases: the intermediate ones represent the core of the proof and to tackle them we use a combination of mirroring, symmetry and second moment arguments to reduce the problem to an easily solvable linear program.



The efficacy of the techniques used in this paper is dependent on the specific values of the $a_i$'s. Our division into different cases allows us to push each of the ideas to their limit. Because of the variety of examples of values $a_i$'s showing the tightness of the conjecture in the sense that $\PP(|\sum_{i=1}^n \epsilon_i a_i| < 1)<\frac{1}{2}$ (e.g. $\frac{1}{3},...,\frac{1}{3}$, and the infinite family $\frac{k-1}{k},\frac{1}{k},...,\frac{1}{k}$ for each $k \geq 2$), it seems inescapable to engage in case analysis. However, the current state of the literature seems to lack this approach.

\section{Set up}

Fix a vector $\mathbf{a}=(a_{1},a_{2},... ,a_{n}) $ with $ \sum_{i=1}^{n} a_{i}^{2} =1 $ and $ a_{1} \geq a_{2} \geq ... \geq a_{n} > 0 $. Let $\epsilon_{i}$ for $i=1,...,n$ be i.i.d. random variables with $\mathbb{P}(\epsilon_{i}=+1)=\mathbb{P}(\epsilon_{i}=-1)= \frac{1}{2}$, i.e. independent \emph{Rademacher random variables}. Denote $\mathbb{P}(\mathbf{a})=\mathbb{P}( |\sum_{i=1}^{n} \epsilon_{i} a_{i}| \leq 1 )$. To show that $\mathbb{P}( \mathbf{a} ) \geq 0.46$, we consider the following four cases depending in which interval $a_{1}$ lies: $[0,0.25],[0.25,0.49],[0.49,0.67]$, and $[0.67,1]$. 


We will use induction on the dimension $n$. Note that for $n=1,2$ the result is trivial. For $n=3$, it follows easily too, by noting that all of the sums $a_{1}-a_{2}+a_{3}, -a_{1}+a_{2}+a_{3},-a_{1}+a_{2}-a_{3}, a_{1}-a_{2}-a_{3}$ have absolute value at most $1$. Thus we will further assume $n \geq 4$. The only time we will appeal to the induction hypothesis is in the proof of \Cref{p1bound}.

We write $\mathbb{P}(N(0,1)\geq x)$ for the probability that a standard normal attains a value of at least $x$.

Several times, we will use the following result of Bentkus and Dzindzalieta \cite{bentkus}.

\begin{lem}\label{auxlem}
Let $ a_{1} \geq a_{2} \geq ... \geq a_{n} > 0 $ be such that $ \sum_{i=1}^{n} a_{i}^{2}  \leq 1 $, and let $\epsilon_{i}$ for $i=1,...,n$ be i.i.d. Rademacher random variables. Then we have for any $x \in \mathbb{R}$

$$\mathbb{P}\left(\sum_{i=1}^{n} \epsilon_{i} a_{i} \geq x\right) \leq 3.18 \ \mathbb{P}(N(0,1) \geq x).$$
\end{lem}


\section{Easy cases - $a_{1}$ small or large}

In this section, we handle the more straightforward cases when either $a_{1}+a_{2} \leq 1, a_{3} \leq 0.25$ or when $a_{1} \geq 0.67$. Here we only need simple mirroring arguments, accompanied by the tail bound provided by \Cref{auxlem}.

\begin{Prop}\label{smalla}
If $a_1+a_2 \le 1$ and $a_3\le 0.25$, then $\mathbb{P}(\mathbf{a})\ge 0.46$.
\end{Prop}

\begin{proof}[Proof of \Cref{smalla}] Define the following random process $ (X_{t})_{t=0}^{n} $. Let $ X_{0}=0 $, and for $1\le t \leq n $, let $X_{t}=\sum_{i=1}^{t} \epsilon_{i} a_{i}$. Let
$$
T=\begin{cases}\text{inf}\{1 \leq t \leq n: |X_{t}|>0.75\} &\text{ if } \{1 \leq t \leq n: |X_{t}|>0.75\}  \neq \emptyset \text{, }\\
n+1 &\text{ otherwise.}\end{cases}
$$
Then $T$ is a stopping time.
Also define random process $ (Y_{t})_{t=0}^{n} $ by setting $Y_{t}=X_{t}$ for $0 \leq t \leq T$ and $ Y_{t}=2X_{T}-X_{t} $ for $n \ge t>T$. Now, $ Y_{n} $ has the same distribution as $X_n=\sum_{i=1}^{n} \epsilon_{i} a_{i}$. 

\begin{clm}\label{protomirror} $\mathbb{P}(|X_{n}|>1 \text{ and } |Y_{n}|>1) < 0.08$.
\end{clm}

\begin{proof}[Proof of \Cref{protomirror}]


Consider the event $|X_n|>1$, and $|Y_n|>1$. We shall show that in this case we have either $|X_n| > 2.5$ or $|Y_n| > 2.5$. By construction it follows that $1 \le T \le n$. Furthermore, we have $0.75 \leq |X_T|\leq 1$, where the upper bound follows from the condition $a_1+a_2\leq 1$ in the case $T=1,2$, and from the condition $ a_3 \leq 0.25$ in the case $3\leq T \leq n$. On the one hand by construction we have $2 \geq 2|X_T|=|X_n+Y_n|$ and on the other hand by assumption we have $2<|X_n|+|Y_n|$. It follows that $|X_n+Y_n| \neq |X_n|+|Y_n|$ which implies that $X_n,Y_n$ have different signs which implies that $|X_n+Y_n|=||X_n|-|Y_n||$.  Therefore, putting all together we have that \begin{align*}
    1.5&\leq 2|X_T|\\
    &=|X_n+Y_n|\\
    &=||X_n|-|Y_n||\\
    &= \max(|X_n|,|Y_n|) - \min(|X_n|,|Y_n|) \\
    &< \max(|X_n|,|Y_n|) -1.
\end{align*}
We get that either $|X_n| > 2.5$ or $|Y_n| > 2.5$. We conclude with the following sequence of inequalities.
\begin{align*}
\mathbb{P}(|X_{n}|>1 \text{ and } |Y_{n}|>1) &\leq \mathbb{P}(|X_{n}|>2.5 \text{ or } |Y_{n}|>2.5) \\
&\leq 2 \, \mathbb{P}(|X_{n}|>2.5) \\
&\leq 6.36 \, \mathbb{P}(|N(0,1)|>2.5)  <0.08,
\end{align*}
where the second inequality follows from the union bound and from the fact that $X_{n},Y_{n}$ have the same distribution and the third inequality follows from \Cref{auxlem}.

\end{proof}

Returning to the proof of the proposition, since $\mathbb{P}(\mathbf{a})=\mathbb{P}(|X_{n}| \leq 1)=\mathbb{P}(|Y_{n}| \leq 1)$, we obtain
\begin{align*}
\mathbb{P}(\mathbf{a})&=\frac{1}{2}\mathbb{P}(|X_{n}| \leq 1)+\frac{1}{2}\mathbb{P}(|Y_{n}| \leq 1) \\
&\geq \frac{1}{2}(1-\mathbb{P}(|X_{n}|>1 \text{ and } |Y_{n}|>1)) \\
&\geq \frac{1}{2}(1-0.08) \\
&=0.46,
\end{align*}

which concludes the proof of \Cref{smalla}.
\end{proof}

\begin{Prop}\label{abig}
If $a_1\ge 0.67$ then $\mathbb{P}(\mathbf{a})\ge 0.46$.
\end{Prop}
\begin{proof}[Proof of \Cref{abig}]
Note that $$\mathbb{P}(\mathbf{a})=\mathbb{P}\left(\left|\sum_{i=1}^{n} \epsilon_{i} a_{i}\right| \leq 1\right) \geq \frac{1}{2} \mathbb{P}\left(\left|\sum_{i=2}^{n} \epsilon_{i} a_{i}\right| \leq 1.67 \right).$$ Consider the unit vector $(b_{2},...,b_{n})$ with $b_{i}=\dfrac{a_{i}}{\sqrt{1-a_{1}^{2}}}$ for $i=2,...,n$, and apply \Cref{auxlem} to conclude that
\begin{align*}
\mathbb{P}\left(\left|\sum_{i=2}^{n} \epsilon_{i} a_{i}\right| \leq 1.67\right) &\geq \mathbb{P}\left(\left|\sum_{i=2}^{n} \epsilon_{i} b_{i}\right| \leq \frac{1.67}{\sqrt{1-0.67^2}} \right)\\
&\geq 1-3.18 \,\mathbb{P}(|N(0,1)|>2.24) \approx 0.9202.
\end{align*}
and hence that $\mathbb{P}(\mathbf{a}) \geq 0.46$.
\end{proof}


So far we resolved the case in which $a_1 \geq 0.67$ and the case in which $a_1+a_2\leq 1$ and $a_3 \leq 0.25$, so it is enough to consider the following two cases:
\begin{itemize}
\item $ 0.25 \leq a_{3} \leq a_{1} \leq 0.49$
\item $0.49 \leq a_{1} \leq  0.67$
\end{itemize}
Each of these cases shall be treated in a separate section.

\section{First intermediate case - $0.25 \leq a_{3} \leq a_{1} \leq 0.49$}

In this section we prove the following proposition.

\begin{Prop}\label{bigone}
If $0.25 \leq a_{3} \leq a_{1} \leq 0.49$, then $\mathbb{P}(\mathbf{a})\ge 0.46$.
\end{Prop}

The strategy is to produce a carefully designed partition of the probability space generated by the possible outcomes of $|\sum_{i\geq 3} \epsilon_i a_i|$. In order to bound the probabilities of these events, the idea is to rely one some mirroring and reflection constructions. Finally, we reduce the problem to an easy linear program.



Assume throughout this section that $0.25 \leq a_{3} \leq a_{1} \leq 0.49$. Let $S=\sum_{i=3}^{n} \epsilon_{i} a_{i}$. Consider the following seven intervals which partition the positive half-line in this order: $I_{1}= [ 0,1-a_{1}-a_{2} ], I_{2}= (1-a_{1}-a_{2},1-a_{1}+a_{2} ], I_{3}= ( 1-a_{1}+a_{2},1+a_{1}-a_{2} ], I_{4}= (1+a_{1}-a_{2},1+a_{1}+a_{2} ], I_{5}= (1+a_{1}+a_{2}, 3-3a_{1}+a_{2} ], I_{6}= (3-3a_{1}+a_{2},3+3a_{1}-5a_{2} ], I_{7}= (3+3a_{1}-5a_{2}, \infty)  $. For $i=1,...,7$, denote $p_{i}= \mathbb{P}(|S| \in I_{i})$. 

Considering the four choices for $(\epsilon_{1},\epsilon_{2})$, by the way this intervals are constructed and by the restrictions on $a_1,a_2,a_3$ we have that 
$$\mathbb{P}\left(\left|\sum_{i=1}^{n} \epsilon_{i} a_{i}\right| > 1 \ \ \Big| \ \ |S| \in I_{j} \right)=\PP\left(|\epsilon_1a_1+\epsilon_2a_2+S|>1 \ \ \Big|\ \ |S|\in I_j\right)=\begin{cases}
0 \text{ if }j=1\\
\frac14 \text{ if } j=2\\
\frac12 \text{ if } j=3\\
\frac34 \text{ if } j=4\\
1\text{ if } j\geq 5.
\end{cases}
$$
Thus we can express
\begin{align}\label{expression}
\mathbb{P}\left(\left|\sum_{i=1}^{n} \epsilon_{i} a_{i}\right| > 1 \right) &= \sum_{j=1}^7 \PP\left(\left|\sum_{i=1}^{n} \epsilon_{i} a_{i}\right| > 1 \ \ \Big| \ \ |S| \in I_{j} \right) \PP(|S|\in I_j)	\\
&=\frac{1}{4}p_{2}+\frac{1}{2}p_{3}+ \frac{3}{4}p_{4}+p_{5}+p_{6}+p_{7}.\nonumber
\end{align}

We shall bound from above this expression, by exploiting various constraints that the $p_{i}$'s satisfy and reducing to a linear program. We collect the constraints into separate lemmas.

Firstly, as the events $\{ |S| \in I_{i} \}$ for $i=1,...,7$ partition our probability space, we know that \begin{equation}\label{fullprob}p_{1}+...+p_{7}=1.\end{equation} Computing the second moment of $S$, we find 
\begin{align}
1-a_{1}^{2}-a_{2}^{2}&=\mathbb{E}(S^{2})\nonumber\\
&=\sum_{i=1}^{7} \PP(|S|\in I_i) \mathbb{E}(S^{2} \, \,| \,\,\, |S| \in I_{i})\nonumber \\
&\geq\sum_{i=1}^{7} p_i (\inf I_i)^2 \label{secmom}\\
&= (1-a_{1}-a_{2})^{2} p_{2} + (1-a_{1}+a_{2})^{2} p_{3} +(1+a_{1}-a_{2})^{2} p_{4} \nonumber \\
&\ \ \ \ 
+ (1+a_{1}+a_{2})^{2} p_{5} +(3-3a_{1}+a_{2})^{2} p_{6}+(3+3a_{1}-5a_{2})^{2} p_{7}\nonumber
\end{align}

\begin{lem}\label{345}
$p_{3}+p_{4}+p_{5} \leq \frac{1}{2}$
\end{lem}
\begin{proof}
Consider the random process $(S_{t})_{t=3}^{n}$, given by $S_{t}=\sum_{i=3}^{t} \epsilon_{i} a_{i}$ for $n \ge t \ge 3$. Let

$$
T_{1}=\begin{cases}\text{inf}\{t \geq 3: |S_{t}|>1-a_{1} \} &\text{ if } \{t \geq 3: |S_{t}|>1-a_{1} \}  \neq \emptyset \text{, }\\
n+1 &\text{ otherwise}\end{cases}
$$
Then $T_{1}$ is a stopping time.
Also define random process $ (U_{t})_{t=3}^{n} $ by setting $U_{t}=S_{t}$ for $3 \leq t \leq T_{1}$ and $ U_{t}=2S_{T_{1}}-S_{t} $ for $n \geq t>T_{1}$. Now,  $ U_{n} $ has the same distribution as $S=S_n$. The conclusion of the claim follows if we show that at most one of $|S_n|, |U_n|$ can lie in the interval $I_{3} \cup I_{4} \cup I_{5}$.

Indeed, if $T_1=n+1$, then $U_n=S_n\in I_1\cup I_2$. Otherwise, if $T_1 \leq n$, then $|S_{T_{1}}|\in(1-a_1,1-a_1+a_2]$. Assume for the sake of contradiction that we have both $|S_n|,|U_n| \in I_3\cup I_{4} \cup I_{5}$. On the one hand, by construction we have $2(1-a_1+a_2) \geq 2|S_{T_1}|=|S_n+U_n|$ and on the other hand, by assumption we have  $2(1-a_1+a_2)< |S_n|+|U_n|$. It follows that $|S_n+U_n| \neq |S_n|+|U_n|$, which implies that $S_n,U_n$ have different signs which implies that $|S_n+U_n|=||S_n|-|U_n||$. Putting all together we have that $$2(1-a_1) \leq 2|S_{T_1}|=|S_n+U_n|=||S_n|-|U_n|| < \sup(I_3\cup I_{4} \cup I_{5})-\inf(I_3\cup I_{4} \cup I_{5}) = 2(1-a_1),$$
which gives the desired contradiction.


\end{proof}

\begin{lem}\label{456}
$p_{4}+p_{5}+p_{6} \leq \frac{1}{2}$
\end{lem}
\begin{proof}
The proof is completely analogous to the proof of previous claim, with the stopping time $T_{2}$ defined by 
$$
T_{2}=\begin{cases}\text{inf}\{t \geq 3: |S_{t}|>1+a_{1}-2a_{2} \} &\text{ if } \{t \geq 3: |S_{t}|>1+a_{1}-2a_{2} \}  \neq \emptyset \text{, }\\
n+1 &\text{ otherwise}\end{cases}
$$

\end{proof}


\begin{lem}\label{p1bound} 
$p_{1} \geq 0.115  \cdot \mathbbm{1}_{a_{1}+a_{2} \leq 0.665}$
\end{lem}

\begin{proof}
 Let
 $$
\widetilde{T}=\begin{cases}\text{inf}\left\{t \geq 4: \left|\sum_{i=4}^{t}\epsilon_{i}a_{i}\right|>0.335 \right\}  &\text{ if } \left\{t \geq 4: \left|\sum_{i=4}^{t}\epsilon_{i}a_{i}\right|>0.335 \right\}  \neq \emptyset \text{, }\\
n+1 &\text{ otherwise}\end{cases}
$$
Then $\widetilde{T}$ is a stopping time.
Further write
$$S=\sum_{i=3}^{n} \epsilon_{i} a_{i}=S_{a}+S_{b}+S_{c} \text{, and }$$
$$S_{(\tau_a, \tau_b, \tau_c)}=\tau_a S_a + \tau_b S_b + \tau_c S_c  \text{ for any } (\tau_a, \tau_b, \tau_c) \in \{\pm 1\}^3,$$
where $S_{a}= \epsilon_{3}a_{3}$, $S_{b}=\sum_{i=4}^{\widetilde{T}} \epsilon_{i}a_{i} $, and $S_{c}=\sum_{i=\widetilde{T}+1}^{n} \epsilon_{i}a_{i}$ if $\widetilde{T}<n$ and $S_{c}=0$ otherwise. Note that $S_{(\tau_a, \tau_b, \tau_c)}$ has the same distribution as $S$.

Assume $a_1+a_2\leq 0.665$ and recall $a_3>0.25$. In order to show that $\mathbb{P}(|S| \leq 1-a_1-a_2) \geq 0.115$ it is enough to show that $ \mathbb{P}(|S| \leq 0.335) \geq 0.115.$


\begin{Obs}\label{twocond}
The conclusion follows if we show that $\mathbb{P}(|S_{c}| \leq 0.91) \geq 0.46$ and that if $|S_c| \leq 0.91$ then there exists $(\tau_a, \tau_b, \tau_c) \in \{\pm 1\}^3$ such that $|S_{(\tau_a, \tau_b, \tau_c)}| \leq 0.335$.
\end{Obs}

Indeed, let $E$ be the event that $|S_c| \leq 0.91$; we have $\mathbb{P}(E) \geq 0.46$. For every point $p \in E$ there exists $(\tau_a^p, \tau_b^p, \tau_c^p) \in \{\pm 1\}^3$ such that $|S_{(\tau_a^p, \tau_b^p, \tau_c^p)}(p)| \leq 0.335$. Note that by construction $|S_{(-\tau_a^p, -\tau_b^p, -\tau_c^p)}(p)| \leq 0.335$. Therefore, there exists $(\tau_a, \tau_b, \tau_c) \in \{\pm 1\}^3$ and an event $F \subset E$ with $\mathbb{P}(F) \geq 0.115$ such that for every point $p \in F$ we have $|S_{(\tau_a, \tau_b, \tau_c)}(p)| \leq 0.335$. As $S_{(\tau_a, \tau_b, \tau_c)}$ has the same distribution as $S$, it follows that $\mathbb{P}(|S| \leq 0.335) \geq 0.115$.

\begin{clm}
$\mathbb{P}(|S_{c}| \leq 0.91) \geq 0.46$.
\end{clm}
\begin{proof}
 Note that the value of $\widetilde{T}$ is independent of the values of $\epsilon_i$ for $i>\widetilde{T}$, so fix a particular value of $\widetilde{T}$. If $\sum_{i=\widetilde{T}+1}^{n}a_{i}=0$, of course the statement is trivial. Otherwise consider the unit vector $(b_i)_{i=\widetilde{T}+1}^{n}$ defined by $b_i=a_i\left(\sqrt{1-\sum_{j=1}^{\widetilde{T}}a_j^2}\right)^{-1}$. By the induction hypothesis applied to this vector, we find
$$\mathbb{P}(|S_{c}| \leq 0.91) \geq\PP\left(|S_c|\leq\sqrt{1-\sum_{j=1}^{\widetilde{T}}a_j^2}\right)= \PP\left(\left|\sum_{i=\widetilde{T}+1}^n \epsilon_ib_i\right|\leq 1\right)\geq 0.46,$$
where the first inequality follows from the fact that $\widetilde{T}\geq4$ and $a_1,a_2,a_3\geq 0.25$.
\end{proof}

\begin{clm} If $|S_c| \leq 0.91$, then there exists $(\tau_a, \tau_b, \tau_c) \in \{\pm 1\}^3$ such that $|S_{(\tau_a, \tau_b, \tau_c)}| \leq 0.335$.
\end{clm}
\begin{proof}
Assume for the sake of contradiction that $|S_{(\tau_a, \tau_b, \tau_c)}| > 0.335$ for all $(\tau_a, \tau_b, \tau_c) \in \{\pm 1\}^3$. Furthermore, assume without loss of generality that $S_{a}, S_{b}, S_{c} \geq 0$. Recall that $S_a=a_3, S_b \in [0, 0.335+a_3], S_c \in [0,0.91]$, that $0.25<a_3\leq \frac{a_1+a_2}{2} \leq 0.3325$ and furthermore that if $S_c>0$, then $S_b \in (0.335, 0.335+a_3]$.

We have $S_a-S_b+S_c \geq a_3-(0.335+a_3)+0=-0.335$ and hence $S_a-S_b+S_c > 0.335$. Similarly, we have $S_a+S_b-S_c \geq -0.335$ by the following dichotomy; if $S_c=0$, then $S_a+S_b-S_c \geq 0.25 + 0 -0 \geq 0.25$, and if $S_c>0$, then $S_a+S_b-S_c \geq 0.25+0.335-0.91>-0.335$. Hence  $S_a+S_b-S_c \geq 0.335$. Combining these inequalities we get $2S_a=2a_3 \geq 0.67$ which contradicts the hypothesis that $a_3 \leq 0.3325$. The conclusion follows.
\end{proof}

The two claims combined with \Cref{twocond} conclude the proof of the \Cref{p1bound}.
\end{proof}

\begin{lem}\label{linprog}
For any parameters $a_{1},a_{2}$ such that $0.25 \leq a_{2} \leq a_{1} \leq 0.49 $, the output $L(a_{1},a_{2})$ of the following linear program satisfies $L(a_{1},a_{2}) \leq 0.54$.
\begin{center}
   $L(a_{1},a_{2}):=\max \{ \frac{1}{4}x_{2}+\frac{1}{2}x_{3}+ \frac{3}{4}x_{4}+x_{5}+x_{6}+x_{7}$ subject to
 \begin{gather*}
 x_1,..., x_7 \geq 0\\
 x_{1}+x_{2}+x_{3}+x_{4}+x_{5}+x_{6}+x_{7}=1\\
 x_{3}+x_{4}+x_{5} \leq \frac{1}{2} \\
 x_{4}+x_{5}+x_{6} \leq \frac{1}{2} \\
 (1-a_{1}-a_{2})^{2} x_{2} + (1-a_{1}+a_{2})^{2} x_{3} +(1+a_{1}-a_{2})^{2} x_{4}  
+ \\ 
+ (1+a_{1}+a_{2})^{2} x_{5} +(3-3a_{1}+a_{2})^{2} x_{6}+(3+3a_{1}-5a_{2})^{2} x_{7} \leq 1-a_{1}^{2}-a_{2}^{2} \\
x_{1} \geq 0.115 \cdot \mathbbm{1}_{a_{1}+a_{2} \leq 0.665}\}
 \end{gather*}
 \end{center}
\end{lem}

\begin{proof}
While we could solve this linear program problem directly, we will instead reduce it to a finite number of cases as follows. Set the margin of error $e=0.005$ and for parameters $a_1',a_2' \in \frac{1}{100} \mathbb{Z}$ of our choice consider the output $L'(a_1',a_2')$ of the following linear program.
\begin{center}
   $L'(a_{1}',a_{2}') := \max \big\{\frac{1}{4}x_{2}+\frac{1}{2}x_{3}+ \frac{3}{4}x_{4}+x_{5}+x_{6}+x_{7}$ subject to
 \begin{gather*}
 x_1,..., x_7 \geq 0 \\
 x_{1}+x_{2}+x_{3}+x_{4}+x_{5}+x_{6}+x_{7}=1\\
 x_{3}+x_{4}+x_{5} \leq \frac{1}{2} \\
 x_{4}+x_{5}+x_{6} \leq \frac{1}{2} \\
(1-a'_{1}-a'_{2}-2e)^{2} x_{2} + (1-a'_{1}+a'_{2}-2e)^{2} x_{3} +(1+a'_{1}-a'_{2}-2e)^{2} x_{4} \\
+ (1+a'_{1}+a'_{2}-2e)^{2}x_{5}   +(3-3a'_{1}+a'_{2}-4e)^{2} x_{6}\\
+(3+3a'_{1}-5a'_{2}-8e)^{2} x_{7} \leq 1-(a'_{1}-e)^{2}-(a'_{2}-e)^{2} \\
x_{1} \geq 0.115 \cdot \mathbbm{1}_{a'_{1}+a'_{2}+2e \leq 0.665} \}.
 \end{gather*}
  \end{center}
  
\begin{Obs}\label{lp1}
If we set $a_1'$ ($a_2'$ resp.) to be $a_1$ ($a_2$ resp.) rounded to the nearest one hundredth then we have $L'(a_1',a_2') \geq L(a_1,a_2)$, as every individual constraint in the linear program $L'$ is at most as strict as its counterpart in the linear program $L$. Given the constraint $0.49 \geq a_{1} \geq a_{2} \geq 0.25$, we deduce the constraint  $0.49 \geq a_1'\geq a_2'\geq 0.25$.
\end{Obs}

A simple computer check shows that for all parameters $a_1',a_2' \in \frac{1}{100}\mathbb{Z}$ that satisfy $0.49 \geq a_1'\geq a_2'\geq 0.25$ we have $L'(a_1',a_2') \leq 0.54$. Using \Cref{lp1} we conclude that $L(a_1,a_2) \leq 0.54$ as desired.
\end{proof}

We conclude this section with the proof of the main proposition. 

\begin{proof}[Proof of \Cref{bigone}]
By \Cref{fullprob}, \Cref{secmom}, \Cref{345}, \Cref{456}, and \Cref{p1bound}, the parameters $p_i$  satisfy the constraints in \Cref{linprog}, so that the set of $\textbf{x}$'s over which $L(a_1,a_2)$ is maximized includes $\textbf{p}$. Finally, by \Cref{expression} and \Cref{linprog}, we conclude that $$\mathbb{P}\left(\left|\sum_{i=1}^{n} \epsilon_{i} a_{i}\right| > 1 \right) =\frac{1}{4}p_{2}+\frac{1}{2}p_{3}+ \frac{3}{4}p_{4}+p_{5}+p_{6}+p_{7} \leq L(a_1,a_2) \leq 0.54 .$$
\end{proof}

\section{Second intermediate case - $0.49 \leq a_{1} \leq 0.67$}

In this section, we solve the last case we have not tackled yet.

\begin{Prop}\label{bigtwo}
If $0.49 \leq a_{1} \leq 0.67$ then $\mathbb{P}(\mathbf{a})\ge 0.46$.
\end{Prop}

We shall follow a similar strategy to the previous section employing the same set of techniques. However, in this section we shall use the linear program only to further reduce the range of vectors $\mathbf{a}$ we are examining. We conclude the remaining cases using additional analytic arguments.

Assume throughout this section that $0.49 \leq a_{1} \leq 0.67$. For $i>1$, we call the term $ a_{i} $ \emph{big} if $ a_{1}+a_{i} > 1 $, and we call it \emph{small} otherwise. 

\begin{lem}\label{anysmall} If we have any small term $ a_{j} $ such that $ a_{j} \geq 0.25 $, then $\mathbb{P}\left(\left|\sum_{i=1}^{n} \epsilon_{i} a_{i}\right| > 1 \right)\leq 0.54$.
\end{lem}
\begin{proof}
Assume we have such an $a_{j}$. Let $U=\sum_{\substack{2\le i \le n \\ i\neq j}} \epsilon_{i} a_{i}$, i.e. the sum of all the signed terms except $a_{1}$ and $a_{j}$. Consider the following five intervals which partition the positive half-line in this order: $I_{1}= [ 0,1-a_{1}-a_{j} ], I_{2}= (1-a_{1}-a_{j},1-a_{1}+a_{j} ], I_{3}= ( 1-a_{1}+a_{j},1+a_{1}-a_{j} ], I_{4}= (1+a_{1}-a_{j},1+a_{1}+a_{j} ], I_{5}= (1+a_{1}+a_{j}, \infty)  $. For $i=1,...,5$, write $p_{i}= \mathbb{P}(|U| \in I_{i})$, so that, analogous to \Cref{expression}, we may write
\begin{equation}\label{expression2}
\mathbb{P}\left(\left|\sum_{i=1}^{n} \epsilon_{i} a_{i}\right| > 1 \right) = 	\frac{1}{4}p_{2}+\frac{1}{2}p_{3}+ \frac{3}{4}p_{4}+p_{5}.
\end{equation}

Analogous to the previous section we get (after noticing the events $\{|U| \in I_i\}$ form a partition of our probability space and after computing the second moment) 
\begin{align}\label{sum2}
    1&=p_1+p_2+p_3+p_4+p_5\\
  \label{secmom2}  1-a_{1}^{2}-a_{j}^{2}&\geq  (1-a_{1}-a_{j})^{2} p_{2} + (1-a_{1}+a_{j})^{2} p_{3}\\
    &\ \ \ \ \ + (1+a_{1}-a_{j})^{2} p_{4} 
+ (1+a_{1}+a_{j})^{2} p_{5}  \nonumber
\end{align}

\begin{clm}\label{linprog2}
For any parameters $a_{1},a_{j}$ such that $0.49 \leq a_{1} \leq 0.67 $ and $0.25\leq a_j\leq 1-a_1$, the output $M(a_{1},a_{j})$ of the following linear program satisfies $M(a_{1},a_{j}) \leq 0.54$.
\end{clm}
\begin{center}
  $M(a_{1},a_{j}) :=\max\{\frac{1}{4}x_{2}+\frac{1}{2}x_{3}+ \frac{3}{4}x_{4}+x_{5}$ subject to
 \begin{gather*}
 x_1,..., x_5 \geq 0 \\
 x_1+x_{2}+x_{3}+x_{4}+x_{5}=1\\
(1-a_{1}-a_{j})^{2} x_{2} + (1-a_{1}+a_{j})^{2} x_{3}+ (1+a_{1}-a_{j})^{2} x_{4}  
+ (1+a_{1}+a_{j})^{2} x_{5}  \leq 1-a_{1}^{2}-a_{j}^{2} \}\\
 \end{gather*}
\end{center}
\begin{proof}
While we could solve this linear program problem directly, we will instead reduce it to a finite number of cases as follows. Set the margin of error $e=0.005$ and for parameters $a_1',a_j' \in \frac{1}{100} \mathbb{Z}$ of our choice consider the output $M'(a_1',a_j')$ of the following linear program.

\begin{center}
  $M'(a_1',a_j'):=\max\{\frac{1}{4}x_{2}+\frac{1}{2}x_{3}+ \frac{3}{4}x_{4}+x_{5}$ subject to
 \begin{gather*}
 x_1,..., x_5 \geq 0 \\
 x_{1}+x_{2}+x_{3}+x_{4}+x_{5}=1\\
g(a'_{1},a'_{j},e)^{2} x_{2} + (1-a'_{1}+a'_{j}-2e)^{2} x_{3}+ \\ +(1+a'_{1}-a'_{j}-2e)^{2} x_{4}  
+ (1+a'_{1}+a'_{j}-2e)^{2} x_{5}  \leq 1-(a'_{1}-e)^{2}-(a'_{j}-e)^{2} \}\\
 \end{gather*}
\end{center}
  
where $g(a'_{1},a'_{j},e)=1-a'_{1}-a'_{j}-2e$ if $1-a'_{1}-a'_{j}-2e>0$, and $g(a'_{1},a'_{j},e)=0$ otherwise.

\begin{Obs}\label{lp2}
If we set $a_1'$ ($a_j'$ resp.) to be $a_1$ ($a_j$ resp.) rounded to the nearest one hundredth then we have $M'(a_1',a_j') \geq M(a_1,a_j)$, as every individual constraint in the linear program $M'$ is at most as strict as its counterpart in the linear program $M$. Given the constraint $0.49 \leq a_{1} \leq 0.67 $ and $0.25\leq a_j\leq \min\{1-a_1,a_1\}$, we deduce the constraint  $0.49 \leq a_1'\leq 0.67$ and $0.25\leq a_j'\leq \min\{1.01-a_1',a_1'\}$.
\end{Obs}

A simple computer check shows that for all parameters $a_1',a_j' \in \frac{1}{100}\mathbb{Z}$ that satisfy $0.49 \leq a_1'\leq 0.67$ and $0.25\leq a_j'\leq \min\{1.01-a_1',a_1'\}$ we have $M'(a_1',a_j') \leq 0.54$. Using \Cref{lp2}, we conclude that $M(a_1,a_j) \leq 0.54$ as desired.


\end{proof}

We return to the proof of the lemma. By \Cref{sum2}, and \Cref{secmom2}, the parameters $p_i$  satisfy the constraints in \Cref{linprog2}, so that the set of $\textbf{x}$'s over which $M(a_1,a_j)$ is maximized includes $\textbf{p}$. Finally, by \Cref{expression2} and \Cref{linprog2}, we conclude that $$\mathbb{P}\left(\left|\sum_{i=1}^{n} \epsilon_{i} a_{i}\right| > 1 \right) =\frac{1}{4}p_{2}+\frac{1}{2}p_{3}+ \frac{3}{4}p_{4}+p_{5} \leq M(a_1,a_j) \leq 0.54 .$$

\end{proof}

\begin{Obs}
It was crucial that $a_{j}$ was a small term. If it was big instead, for the interval 
$I_1'=[0,a_{1}+a_{j}-1 )$ around the origin, we have

\begin{equation*}
\PP\left(\left|\sum_{i=1}^{n} \epsilon_{i} a_{i}\right| > 1 \   \Big|\ |U|\in I_1' \right)=\frac12.
\end{equation*}
This is in contrast with $\PP(|\sum_i \epsilon_ia_i|>1 \,\, \big| \,\, |U|\in I_1)=0$, which we used in the proof of \Cref{anysmall}.
\end{Obs}

Henceforth we shall assume that there exist no small terms of size at least $0.25$ and we shall use a mirroring argument similar to the one we used in Section 3 to conclude. Let $k$ be such that the terms $a_{2},...,a_{k}$ are big and the terms $a_{k+1},...,a_{n}$ are small. We will need the following easy lemma.

\begin{lem}\label{sumlem}
 If $2 \leq l \leq k$, then we have $a_{2}+a_{3}+...+a_{l-1}+2a_{l} \leq 2$.
\end{lem}
\begin{proof}
 Using the fact that $\sum_{i=2}^{l}a_{i}^{2} \leq 1-a_{1}^{2}$ and that $a_{l}$ is the smallest term out of $a_{2},...,a_{l}$, we get $a_{l} \leq \sqrt{\frac{1-a_{1}^{2}}{l-1}}$. Using Cauchy-Schwarz inequality, we get
\begin{align*}
a_{2}+a_{3}+...+a_{l-1}+2a_{l} &\leq (a_{2}+...+a_{l}) + \sqrt{\frac{1-a_{1}^{2}}{l-1}} \\
&\leq  \sqrt{(l-1)(a_{2}^{2}+...+a_{l}^{2})} + \sqrt{\frac{1-a_{1}^{2}}{l-1}} \\
&\leq \sqrt{l-1}\sqrt{1-a_{1}^{2}} + \sqrt{\frac{1-a_{1}^{2}}{l-1}}
\end{align*}

Next, note that as each big term is bigger than $1-a_{1}$, so
$$1-a_1^2\geq a_2^2+a_3^2+\dots+a_l^2\geq (l-1)(1-a_1)^2$$
and thus $l-1\leq \frac{1-a_{1}^{2}}{(1-a_{1})^{2}}=\frac{1+a_{1}}{1-a_{1}}$. 

Combining these two with the fact that the function $x+\frac{1}{x}$ is increasing on the interval $[1,\infty) $, we find

\begin{align*}
\sqrt{l-1}\sqrt{1-a_{1}^{2}} + \sqrt{\frac{1-a_{1}^{2}}{l-1}}  &\leq \sqrt{1-a_{1}^{2}} \left(\sqrt{\frac{1+a_{1}}{1-a_{1}}}+\sqrt{\frac{1-a_{1}}{1+a_{1}}}\right) \\
 &\leq (1+a_{1})+(1-a_{1}) =2
\end{align*}

This concludes the proof of the lemma.
\end{proof}

\begin{proof}[Proof of \Cref{bigtwo}] Define the following random process $ (A_{t})_{t=0}^{n} $. We set $ A_{0}=0 $, $A_{1}=\epsilon_{1} a_{1}$ and for $ n\geq t \geq 2 $, $A_{t}=\epsilon_{1} a_{1}+\sum_{i=n-t+2}^{n} \epsilon_{i} a_{i}$. Let
$$
T=\begin{cases}\text{inf}\{1 \leq t \leq n: |A_{t}|>1-a_{n-t+1}\} &\text{ if } \{1 \leq t \leq n: |A_{t}|>1-a_{n-t+1}\}  \neq \emptyset \text{, }\\
n+1 &\text{ otherwise}\end{cases}
$$
Then $T$ is a stopping time.
Note that if $T\leq n$, then $|A_T|\leq1$. Also define the random process $ (B_{t})_{t=0}^{n} $ by setting $B_{t}=A_{t}$ for $t \leq T$ and $ B_{t}=2A_{T}-A_{t} $ for $n\ge t>T$. Note that $ B_{n} $ has the same distribution as $A_n=\sum_{i=1}^{n} \epsilon_{i} a_{i}$. 

\begin{clm} If $|A_{n}|>1 \text{ and } |B_{n}|>1$, then $|A_{n}|>2.5$ or $|B_{n}|>2.5$.
\end{clm}

\begin{proof}
Assume $|A_n|,|B_n|>1$. Clearly $T\leq n-1 $ as otherwise if $T=n,n+1$, then by construction we have $|A_n|,|B_n|\leq1$. Now for $T \leq n-1$, note that we have $|A_T|\leq 1$ and hence 
\begin{equation*}
|A_n|+|B_n|>2 \geq 2 |A_T| = |A_n+B_n|    
\end{equation*}
It follows that $A_n$ and $B_n$ must have opposite signs.

We argue $T < n-k+1$. Indeed, assume for the sake of contradiction that $ n-k+1\leq T<n$, and furthermore assume that $A_{T}>1-a_{n-T+1}$. As $n-T+1\leq k$, by \Cref{sumlem} we have that
$$A_n,B_n\geq A_{T}-(a_{2}+...+a_{n-T+1}) > 1-(a_{2}+...+2a_{n-T+1}) \geq -1.$$ 
This gives the desired contradiction as $A_{n}$ and $B_{n}$ have modulus strictly greater than $1$ and opposite signs.

For $T < n-k+1$, we get that $a_{n-T+1}$ is a small term, so $|A_{T}|>0.75$. As $A_n,B_n$ have opposite signs we have $|A_n+B_n|=||A_n|-|B_n||$. Therefore, putting all together we have that \begin{align*}
    1.5&\leq 2|A_T|\\
    &=|A_n+B_n|\\
    &=||A_n|-|B_n||\\
    &= \max(|A_n|,|B_n|) - \min(|A_n|,|B_n|) \\
    &< \max(|A_n|,|B_n|) -1.
\end{align*}
This concludes the claim.
\end{proof}

Similarly to the proof of \Cref{smalla}, we now have
\begin{align*}
\mathbb{P}(|A_{n}|>1 \text{ and } |B_{n}|>1) &\leq \mathbb{P}(|A_{n}|>2.5 \text{ or } |B_{n}|>2.5) \\
&\leq 2 \, \mathbb{P}(|A_{n}|>2.5) \\
&\leq 6.36 \, \mathbb{P}(|N(0,1)|>2.5)  <0.08,
\end{align*}
where the second inequality follows from the union bound and from the fact that $A_{n},B_{n}$ have the same distribution and the third inequality follows from \Cref{auxlem}.

We conclude that, since $\mathbb{P}(\mathbf{a})=\mathbb{P}(|A_{n}| \leq 1)=\mathbb{P}(|B_{n}| \leq 1)$, we obtain
\begin{align*}
\mathbb{P}(\mathbf{a})&=\frac{1}{2}\mathbb{P}(|A_{n}| \leq 1)+\frac{1}{2}\mathbb{P}(|B_{n}| \leq 1) \\
&\geq \frac{1}{2}(1-\mathbb{P}(|A_{n}|>1 \text{ and } |B_{n}|>1)) \\
&\geq \frac{1}{2}(1-0.08) \\
&=0.46.
\end{align*}
This finishes the proof of \Cref{bigtwo}.
\end{proof}

We conclude this section, and thus also the entire proof with some remarks. We believe that with the ideas presented here, by doing a more careful analysis in which one considers a more refined partition of the parameter space, the bound of $0.46$ could probably be further improved. However, with the current partition into cases, the bound that we get is close to optimal. Hence, to prove the full conjecture with the bound of $0.5$, new ideas will be needed.





\section*{Acknowledgements}

We would like to thank our PhD supervisor B\'ela Bollob\'as for advice regarding the final version of this note.

\end{document}